\documentclass{amsart}
\usepackage{nz_style} 
\usepackage{chngcntr}
\counterwithin{equation}{section}
\title{Borel's rank theorem for Artin $L$-functions}
\author{
	Ningchuan Zhang
}
\address{Department of Mathematics, University of Pennsylvania, Philadelphia, PA 19104, USA}
\hemail{nczhang@sas.upenn.edu}

\begin{document}
	\begin{abstract}
		Borel's rank theorem identifies the ranks of algebraic $K$-groups of the ring of integers of a number field with the orders of vanishing of the Dedekind zeta function attached to the field. Following the work of Gross, we establish a version of this theorem for Artin $L$-functions by considering equivariant algebraic $K$-groups of number fields with coefficients in rational Galois representations. This construction involves twisting algebraic $K$-theory spectra with rational equivariant Moore spectra. We further discuss integral equivariant Moore spectra attached to Galois representations and their potential applications in $L$-functions. 
	\end{abstract}
	\maketitle
	\section{Introduction}
	Let $\Fbb$ be a number field and $\Ocal_\Fbb$ be its ring of integers. The Dedekind  zeta function attached to $\Fbb$ is defined to be:
	\begin{equation*}
		\zeta_\Fbb(s)=\sum_{(0)\ne \Ical\trianglelefteq \Ocal_\Fbb}\frac{1}{|\Ocal_\Fbb/\Ical|^s}=\prod_{(0)\ne \mathfrak{p}\trianglelefteq \Ocal_\Fbb}\frac{1}{1-|\Ocal_\Fbb/\mathfrak{p}|^{-s}},
	\end{equation*}
	where $\Ical$ ranges over all non-zero ideals of $\Ocal_\Fbb$ in the summation, and $\mathfrak{p}$ ranges over all non-zero prime (maximal) ideals of $\Ocal_\Fbb$ in the product. This summation converges when $\Re(s)>1$ and admits an analytic continuation to $\Cbb\backslash\{1\}$. The orders of vanishing of $\zeta_\Fbb(s)$ at non-positive integers carry the following algebraic information:
	\begin{thm}[Quillen, Borel, {\cite{Quillen_fin_gen,Borel_cohomologie,Borel_stable_real_cohomology}}]\label{thm:QB}
		Algebraic $K$-groups of $\Ocal_\Fbb$ are all finitely generated abelian groups. In even degrees, $K_{2n}(\Ocal_\Fbb)$ is a finite abelian group when $n\ge 1$. Denote by $r_1$ and $r_2$ the numbers of real and conjugate pairs of complex places of $\Fbb$, respectively. In odd degrees, we have
		\begin{align*}
			\dim_\Q K_{2n-1}(\Ocal_\Fbb)\otimes \Q=\dim_\Q \pi_{2n-1}\left(K(\Ocal_\Fbb)\wedge M(\Q)\right)&=\left\{\begin{array}{cl}
				r_1+r_2-1, &n=1;\\
				r_1+r_2, & n>1\text{ odd};\\
				r_2,& n\text{ even},\\
			\end{array}\right.\\&=\ord_{s=1-n}\zeta_\Fbb(s),
		\end{align*}
		where $K(\Ocal_{\Fbb})$ is the algebraic $K$-theory \emph{spectrum} of $\Ocal_{\Fbb}$, and $M(\Q)$ is the Moore spectrum of $\Q$.
	\end{thm}
	The main goal of this paper is to generalize \Cref{thm:QB} to \textbf{Artin $L$-functions}. 
	\begin{defn}\label{defn:artin_L}
		 Consider a finite $G$-Galois extension of number fields $\Fbb/\Kbb$. Let $\rho\colon G\to \aut_\Lbb(V)$ be a representation of $G$ valued in a finite dimensional vector space $V$ over a number field $\Lbb$. The Artin $L$-function attached to $\rho$ is defined to be the Euler product:
		\begin{equation*}
			L(s;\rho)=\prod_{\mathfrak{p}}\frac{1}{\det\left[\id-|\Ocal_{\Kbb}/\mathfrak{p}|^{-s}\cdot \rho(\mathrm{Frob}_\mathfrak{p})\left| V^{\mathfrak{I_p}}\right.\right]}.
		\end{equation*}
		In this formula:
		\begin{itemize}
			\item $\mathfrak{p}$ ranges over all non-zero prime ideals of $\Ocal_{\Kbb}$.
			\item $\mathfrak{I_p}$ is the inertia subgroup of $G$ at a prime $\mathfrak{P}$ in $\Fbb$ over $\mathfrak{p}$. Different choices of $\mathfrak{P}$ result in conjugate inertia subgroups $\mathfrak{I_p}$, which does not affect the definition. 
			\item $\mathrm{Frob}_\mathfrak{p}$ is a Frobenius element.
		\end{itemize}
		See full details of the definition in \cite{Murty_Artin_L}.
	\end{defn}
	\begin{exmps}
		\begin{enumerate}
			\item When $\rho$ is the $1$-dimension trivial representation of $G=\gal(\Fbb/\Kbb)$, the Artin $L$-function $L(s,\rho)$ is the Dedekind zeta function of $\Kbb$.
			\item When $\Fbb/\Kbb=\Q(\zeta_N)/\Q$ is the $N$-th cyclotomic extension of $\Q$, and $\rho=\chi\colon \znx\cong \gal(\Q(\zeta_N)/\Q)\to \Cx$ is a Dirichlet character, we recover the \textbf{Dirichlet $L$-function} $L(s,\chi)$ attached to $\chi$:
			\begin{equation*}
				L(s,\chi)=\sum_{n=1}^{\infty}\frac{\chi(n)}{n^s},\quad \chi(n)=0 \text{ if }\gcd(n,N)\ne 1.
			\end{equation*}
		\end{enumerate}
	\end{exmps}
	Building on an algebraic Borel's rank \Cref{thm:alg_QB_AL} for Artin $L$-functions by Gross, we will prove:
	\begin{thm}\label{thm:main}
	Consider the Artin $L$-function $L(s,\rho)$ attached to a Galois representation $\rho\colon G=\gal(\Fbb/\Kbb)\to \aut_\Lbb(V)$ as in \Cref{defn:artin_L}. Denote the $G$-equivariant Moore spectrum associated to $\rho$ by $M(\underline{\rho})$. Then the orders of vanishing of $L(s,\rho)$ at non-positive integers are computed by the dimensions of equivariant homotopy groups:
		\begin{equation*}
			\ord_{s=1-n}L(s,\rho)=\dim_\Lbb\pi_{2n-1}\left[\left(K(\Ocal_{\Fbb})\wedge M\left(\underline{\rho}\right)\right)^{hG}\right].
		\end{equation*}	
	\end{thm}
	\begin{rem}\label{rem:htpy_coeff}
		Let $A$ be an abelian group and $M(A)$ be its Moore spectrum (\Cref{defn:Moore}). In stable homotopy theory, homotopy groups of a spectrum $X$ \emph{with coefficients in $A$} are defined to be:
		\begin{equation*}
			\pi_n(X;A)=\pi_n(X\wedge M(A)).
		\end{equation*}
		In this way, \Cref{thm:main} can be translated to saying equivariant algebraic $K$-groups of the ring of integers of a number field $\Fbb$ \emph{with coefficients in a rational Galois representation} $\rho$ compute the orders of vanishing of the corresponding Artin $L$-function $L(s,\rho)$ at non-positive integers. 
	\end{rem}
	The organization of this paper is as follows:
	\begin{itemize}
		\item In \Cref{sec:background}, we review Galois group actions on rational algebraic $K$-groups of number fields (\Cref{thm:QB_gal}), and state an \emph{algebraic} Borel's rank theorem for Artin $L$-functions (\Cref{thm:alg_QB_AL}) by Gross in \cite{Gross_Artin_L}.
		\item  In \Cref{sec:rational_alg_K}, we will study the rational equivariant algebraic $K$-theory of number fields. In the first half of the section, we will construct a spectral lifting of the equivariant Borel regulator maps in \Cref{thm:spectral_Borel_reg} and give explicit descriptions of the rational equivariant homotopy type of the algebraic $K$-theory spectrum $K(\Ocal_{\Fbb})$ with respect to the Galois group $\gal(\Fbb/\Q)$ in \Cref{prop:G-cofib} and \Cref{exmp:emb_gal_c2}. \Cref{thm:main} will be proved in the second half of this section. 
		\item \Cref{thm:main} relates \emph{rational} equivariant algebraic $K$-theory of number fields to Artin $L$-function. In \Cref{sec:int_eMoore}, we consider a potential \emph{integral} version of this connection. Non-equivariantly, the Quillen-Lichtenbaum Conjecture \Cref{thm:QLC}, proved by Voevodsky-Rost, connects special values of Dedekind zeta functions $\zeta_\Fbb(s)$ to torsion subgroups of $K_{*}(\Ocal_{\Fbb})$. To generalize this to Artin $L$-functions, a first obstruction is Steenrod's \Cref{quest:Steenrod} on the existence of integral equivariant Moore spectra attached to Galois representations. My previous 
		work in \cite{nz_Dirichlet_J} implies that integral equivariant Moore spectra attached to abelian characters of finite groups always exist (\Cref{cor:eMoore_char}). In my 
		current joint work in progress with Elden Elmanto, we will study a potential Quillen-Lichtenbaum Conjecture for Dirichlet $L$-functions using those integral equivariant Moore spectra attached Dirichlet characters. 
	\end{itemize}

	\subsection*{Acknowledgments} 
	I would like to thank Elden Elmanto,  Mona Merling, and Maximilien P\'eroux for their comments and suggestions after carefully reading through earlier drafts of this paper. 
	I also want to thank Matt Ando, Mark Behrens, Ted Chinburg,  Guchuan Li,  and Charles Rezk for helpful discussions related to this project.  Finally, I would like to thank the anonymous referee for many helpful comments and suggestions on revisions.
	\section{Background: Galois actions on rational algebraic $K$-groups}\label{sec:background}
	 To prove \Cref{thm:main}, we first recall Galois actions on rational algebraic $K$-groups of $\Ocal_\Fbb$, described by Gross in \cite{Gross_Artin_L}. Let $\emb(\Fbb)$ be the set of field embeddings $\Fbb\hookrightarrow \Cbb$. Denote the cyclic group of order $2$ by $C_2$ with generator $c$. Then $\emb(\Fbb)$ is a $C_2$-$G$-set, where the Galois group $G$ acts by pre-compositions and $C_2$ acts by complex conjugation. 
	Notice for any embedding $v\colon\Fbb \hookrightarrow \Cbb$ and $g\in G$, we have $c\cdot (v\circ g)=(c\cdot v)\circ g$. This implies the $(-1)^n$-eigenspace of the $C_2$-action
	\begin{equation*}
		Y_n(\Fbb)=\left\{\left.v\in \Q^{\emb(\Fbb)}\right| c\cdot v=(-1)^n v\right\}
	\end{equation*}
	is a rational $G$-representation. In proving \Cref{thm:QB}, Dirichlet (for $n=1$) and Borel (for $n>1$) constructed the regulator maps 
	\begin{equation}\label{eqn:Borel_reg}
		\begin{tikzcd}
			R^n_\Fbb\colon K_{2n-1}(\Ocal_\Fbb)\rar& Y_{n-1}(\Fbb)\otimes_\Q\Rbb.
		\end{tikzcd}
	\end{equation}
	Those maps are described in \cite{Gross_Artin_L} as follows. Each embedding $\iota\colon \Fbb\hookrightarrow \Cbb$ induces a map in algebraic $K$-groups, assembling into a Galois-equivariant map that factors through $C_2$-fixed points
	\begin{equation}\label{eqn:assemble_emb}
		\lambda_{2n-1}\colon K_{2n-1}(\Ocal_{\Fbb})\longrightarrow \map^{C_2}(\emb(\Fbb),K_{2n-1}(\Cbb)).
	\end{equation}
	Let $\Rbb(m)$ be the $1$-dimensional real $C_2$-representation with $c\in C_2$ acting by multiplication by $(-1)^m$.  By identifying rational $K$-theory of a ring $A$ with primitives in the rational homology groups of $\GL(A)$, we obtain a family of $C_2$-equivariant maps, called the universal Borel regulators \cite[page 134]{CMS-GS-2015}:
	\begin{equation}\label{eqn:KC_reg}
		e_{2n-1}\colon K_{2n-1}(\Cbb)\to \Rbb(n-1).
	\end{equation}
	The regulator map $R^n_\Fbb$ in \eqref{eqn:Borel_reg} is defined to be the composition $(e_{2n-1})_*\circ \lambda_{2n-1}$.
	
	When $n>1$, the image of the regulator  map is a lattice in the real vector space $Y_{n-1}(\Fbb)\otimes_\Q\Rbb$. In the $n=1$ case, the image of the regulator map is a lattice in a hyperplane $\overline{Y_0(\Fbb)}\otimes_\Q\Rbb\subseteq Y_0(\Fbb)\otimes_\Q\Rbb$. This hyperplane can be viewed as the kernel of a $G$-equivariant map 
	\begin{equation}\label{eqn:hyperplane}
		f\colon Y_0(\Fbb)\otimes_\Q\Rbb \longrightarrow \Rbb_\triv,
	\end{equation}	
	where $\Rbb_\triv$ denotes $\Rbb$ with the trivial $G$-action.
	\begin{thm}[{\cite[Theorems 2.4 and 2.5]{Gross_Artin_L}}] \label{thm:QB_gal}
		The Dirichlet and Borel regulator maps are Galois equivariant, inducing isomorphisms of rational $G$-representations:
		\begin{equation*}
			K_{m}(\Ocal_\Fbb)\otimes \Q\cong\left\{\begin{array}{cl}
				\Q_\triv,& m=0;\\
				\overline{Y_0(\Fbb)}, & m=1;\\
				Y_{n-1}(\Fbb), & m=2n-1>1;\\
				0, & \text{else}.
			\end{array}\right.
		\end{equation*}
	\end{thm}
	Using \Cref{thm:QB_gal} and Frobenius reciprocity, Gross proved:
	\begin{thm}[Gross, {\cite[Equation 3.9]{Gross_Artin_L}}]\label{thm:alg_QB_AL}
		Consider the Artin $L$-function $L(s,\rho)$ attached to a Galois representation $\rho\colon G=\gal(\Fbb/\Kbb)\to \aut_\Lbb(V)$ as in \Cref{defn:artin_L}. Denote the $\Lbb$-vector space $V$ with the associated $G$-action by $\underline{\rho}$. Then 
		\begin{equation*}
			\ord_{s=1-n}L(s,\rho)=\dim_\Lbb\left[K_{2n-1}(\Ocal_\Fbb)\otimes_\Z \underline{\rho}\right]^G.
		\end{equation*}
	\end{thm}	
	\section{Rational equivariant algebraic $K$-theory spectra of number fields}\label{sec:rational_alg_K}
	\Cref{thm:alg_QB_AL} is an algebraic Borel's rank theorem for Artin $L$-functions. To prove \Cref{thm:main}, we need to identify the \emph{$G$-fixed point subspaces} in \Cref{thm:alg_QB_AL} with the \emph{$G$-equivariant homotopy groups} in \Cref{thm:main}. Algebraic $K$-groups of a commutative ring $R$ are homotopy groups of the algebraic $K$-theory \emph{spectrum} $K(R)$ of $R$. For a commutative ring $R$ with a $G$-action, its algebraic $K$-theory spectrum $K(R)$ has been constructed a \emph{genuine} $G$-spectrum by Merling and Barwick-Glasman-Shah in \cite{Merling_equivariant_K,Barwick_spectral_Mackey_I, BGS_spectral_Mackey_II}. For simplicity, we will consider $K(R)$ only as a \emph{na\"ive} $G$-spectrum, i.e. a spectrum with a $G$-action, in this paper. The first step to prove \Cref{thm:main} is to rationalize the algebraic $K$-theory spectra of number fields. As mentioned in \Cref{rem:htpy_coeff}, rational algerbaic $K$-groups of a ring $R$ are defined to be homotopy groups of $K(R)\wedge M(\Q)$, where $M(\Q)$ is the Moore spectrum for the rational numbers. 
	\begin{defn}\label{defn:Moore}
		Let $A$ be an abelian group. The \textbf{Moore spectrum} of $A$ is the connective spectrum $M(A)$ such that
		\begin{equation*}
			H_*(M(A);\Z)=\left\{\begin{array}{cl}
				A, & *=0;\\
				0,& \text{else.}
			\end{array}\right.
		\end{equation*}  
	\end{defn}
	\begin{rem}
		Here are some basic facts about Moore spectra:
		\begin{enumerate}
			\item While the definition above uniquely determines the spectrum $M(A)$ up to weak equivalences, the assignment $A\mapsto M(A)$ is \emph{not} functorial. 
			\item From the additivity axiom of singular homology, we have $M(A\oplus B)\simeq M(A)\vee M(B)$ for any abelian groups $A$ and $B$.
			\item The Moore spectrum construction is \emph{not} monoidal in general. Using the K\"unneth Theorem for singular homology groups, one can check that $M(\Z/2)\wedge M(\Z/2)\not\simeq M(\Z/2),$ even though $\Z/2\otimes_\Z \Z/2\cong \Z/2$.
		\end{enumerate}
	\end{rem}
	\begin{prop}\label{prop:rational_ME}
		Let $A$ be a $\Q$-vector space. Denote its Eilenberg-MacLane spectrum by $HA$. Then the map $MA\to HA$ classifying $\mathrm{id}_A\in \hom(A,A)\cong [M(A),HA]$ is an equivalence. Consequently, for any field $\Lbb$ of characteristic zero, its Moore spectrum $M(\Lbb)\simeq H\Lbb$ is an $E_\infty$-ring spectrum. In addition, the Moore spectrum $M(V)$ for any $\Lbb$-vector space $V$ has a natural $H\Lbb$-module spectrum structure. 
	\end{prop}
	\begin{proof}
		This is a standard result in rational homotopy theory, which essentially follows from a result of Serre (see below). We first show $M(A)\simeq HA$ for a $\Q$-vector space $A$. The universal coefficient  theorem for stable homotopy groups and the flatness of $A$ as a $\Z$-module implies:
		\begin{equation*}
			\pi_*(M(A))\cong \pi_*(S^0)\otimes_\Z A.
		\end{equation*}
		It now follows from Serre's theorem on the finiteness of stable homotopy groups of spheres that 
		\begin{equation*}
			\pi_*(S^0)\otimes_\Z A\cong \left\{\begin{array}{cl}
				A, & *=0;\\
				0, &\text{else.}
			\end{array}\right.
		\end{equation*}
		The second part of the claim follows by applying the lax-monoidal Eilenberg-MacLane functor $H$ to the structure map $\Lbb\otimes_\Z V\to V$ of $V$ as an $\Lbb$-vector space.
	\end{proof}
	\begin{lem}\label{lem:rational_FPT}
		Let $G$ be a finite group and $X$ be a na\"ive $G$-spectrum with rational homotopy groups. Then we have an isomorphism of abelian groups. 
		\begin{equation*}
			\pi_*\left(X^{hG}\right)\cong \left[\pi_*(X)\right]^G.
		\end{equation*}
	\end{lem}
	\begin{proof}
		This follows from the fact that the group cohomology of a finite group $G$ with rational coefficients vanishes in positive degrees. As a result, the homotopy fixed point spectral sequence
		\begin{equation*}
			E_2^{s,t}=H^s(G;\pi_t(X))\Longrightarrow \pi_{t-s}(X^{hG})
		\end{equation*} 
		is concentrated in the $(s=0)$-line, yielding an isomorphism  $\pi_*(X^{hG})\cong [\pi_*(X)]^G$.
	\end{proof}
	\Cref{thm:QB_gal} can now be translated into a description of the rational $G$-equivariant homotopy type of the algebraic $K$-theory spectrum $K(\Ocal_{\Fbb})$. First, we lift the Dirichlet/Borel regulator maps on algebraic $K$-groups described in \Cref{sec:background} to an equivariant map between na\"ive rational $G$-spectra.
	\begin{thm}\label{thm:spectral_Borel_reg}
		Let $\Fbb$ be a number field and $G=\gal(\Fbb/\Q)$ be Galois group of $\Fbb$ over $\Q$, i.e. the group of field automorphisms of $\Fbb$ over $\Q$. Denote by $kr$ the connective cover of Atiyah's $C_2$-equivariant Real $K$-theory spectrum.  For a topological space $X$, denote its suspension spectrum by $\Sigma_+^{\infty}X$. Then there is a $G$-equivariant map of spectra 
		\begin{equation}
			\begin{tikzcd}
				R_\Fbb\colon K(\Ocal_{\Fbb})\rar["\lambda"]&\map\left(\Sigma_+^{\infty}\emb(\Fbb), K(\Cbb)\right)^{hC_2}\rar["e_*"] & \map\left(\Sigma_+^{\infty}\emb(\Fbb), \Sigma kr\wedge H\Rbb\right)^{hC_2},
			\end{tikzcd}
		\end{equation}
		whose induced map on $\pi_{2n-1}$ is the Dirichlet/Borel regulator map $R^n_{\Fbb}$. 
	\end{thm}
	\begin{proof}
		Following the description of $G$-equivariant Dirichlet and Borel regulator maps in \Cref{sec:background}, it suffices to lift the maps $\lambda_{2n-1}$ in \eqref{eqn:assemble_emb} and $e_{2n-1}$ in \eqref{eqn:KC_reg} to equivariant maps between na\"ive $G$-spectra.
		
		To lift $\lambda_{2n-1}$, consider the adjoint of the evaluation map:
		\begin{equation*}
			\begin{tikzcd}[row sep=0, /tikz/column 1/.append style={anchor=base east}
				,/tikz/column 2/.append style={anchor=base west}]
				v\colon \Ocal_{\Fbb}\rar[hook] & \map\left(\emb(\Fbb),\Cbb\right),\\
				x\rar[mapsto] &(\iota\mapsto \iota(x)).
			\end{tikzcd}
		\end{equation*}
		This is a $G$-equivariant ring homomorphism, where $G$ acts trivially on $\Cbb$. The homomorphism $v$ induces a map $K(v)$ on algebraic $K$-theory spectra:
		\begin{equation*}
			\begin{tikzcd}
				K(\Ocal_\Fbb)\rar["K(v)"]\dar["\lambda"',dashed]& K\left[\map\left(\emb(\Fbb),\Cbb\right) \right]\dar["\simeq"]\\ \map\left(\Sigma_+^{\infty}\emb(\Fbb),K(\Cbb)\right)^{hC_2}\rar&\map\left(\Sigma_+^{\infty}\emb(\Fbb),K(\Cbb)\right)
			\end{tikzcd}			
		\end{equation*}
		The equivalence on the right exists because $\emb(\Fbb)$ is a finite $G$-set and algebraic $K$-theory commutes with finite products. It is $G$-equivariant since $G$ acts trivially on $\Cbb$. Notice $C_2$ acts trivially on $K(\Ocal_\Fbb)$, the map $K(v)$ factors through the $C_2$-homotopy fixed points of the target. The resulting map $\lambda$ lifts the maps $\lambda_{2n-1}$'s on algebraic $K$-groups. 
		
		It now remains to construct a $C_2$-equivariant map $e\colon K(\Cbb)\to \Sigma kr\wedge H\Rbb$ whose induced maps on homotopy groups are the maps $e_{2n-1}$ in \eqref{eqn:KC_reg}. Recall the generator $c\in C_2$ acts on $\pi_{2k}(kr)$ by multiplication by $(-1)^k$.  By \cite[Theorem A.1]{GM_gen_Tate}, rational $G$-spectra for a finite group $G$ are wedge sums of suspensions of equivariant Eilenberg-MacLane spectra.  This yields a $C_2$-equivalence:
		\begin{equation*}
			\Sigma kr\wedge H\Rbb\simeq \bigvee_{n=1}^\infty \Sigma^{2n-1} H\Rbb(n-1).
		\end{equation*}
		It follows that a map $e\colon K(\Cbb)\to \Sigma kr\wedge H\Rbb$ is determined up to homotopy by the universal Borel regulators $\{e_{2n-1}\}$. The latters are identified with indecomposable $C_2$-equivariant cohomology classes in 
		\begin{equation*}
			H^{2n-1}_{C_2}(K(\Cbb),\Rbb(n-1))\cong [K(\Cbb),\Sigma^{2n-1}H\Rbb(n-1)]_{C_2}.\qedhere
		\end{equation*}
	\end{proof}
	\begin{rem}
		In \cite{BNT_Beilinson_reg}, Bunke-Nikolaus-Tamme constructed a spectral lifting of the \textbf{Beilinson regulators} as a morphism between $E_\infty$-motivic spectra over $\Cbb$
		\begin{equation*}
			R_\mathrm{Beilinson}\colon \mathbf{K}\longrightarrow \mathbf{H},
		\end{equation*}
		where $\mathbf{K}$ and $\mathbf{H}$ represent algebraic $K$-theory and absolute Hodge cohomology of smooth complex algebraic varieties, respectively. Burgos Gil showed in \cite[Theorem 10.9]{BurgosGil_regulators} that the universal Beilinson and Borel regulators (see \eqref{eqn:KC_reg}) are related by:
		\begin{equation*}
			e^\mathrm{Borel}_{2n-1}=2\cdot e^\mathrm{Beilinson}_{2n-1}\in H^{2n-1}_{C_2}(K(\Cbb),\Rbb(n-1)).
		\end{equation*}
	\end{rem}
	Using the spectral equivariant Borel regulator map $R_\Fbb$ above, we can read off the rational (real) $G$-equivariant homotopy type of $K(\Ocal_{\Fbb})$.
	\begin{prop}\label{prop:G-cofib}
		There is a $G$-equivariant cofiber sequence:
		\begin{equation*}
			\begin{tikzcd}[column sep=large]
				K(\Ocal_{\Fbb})\wedge H\Rbb \rar["R_\Fbb\wedge 1"] & \map\left(\Sigma_+^{\infty}\emb(\Fbb), \Sigma kr\wedge H\Rbb\right)^{hC_2}\rar["\phi\vee 0"] & \Sigma H\Rbb\vee \Sigma H\Rbb, 
			\end{tikzcd}
		\end{equation*}
		where 
		\begin{itemize}
			\item $R_\Fbb\wedge 1$ is the  $H\Rbb$-module map adjoint to $R_\Fbb$;
			\item $\pi_1(\phi)$ is the map $f$ in \eqref{eqn:hyperplane};
			\item $G$ acts trivially on the cofiber $\Sigma H\Rbb\vee \Sigma H\Rbb$.
		\end{itemize}
	\end{prop}
	\begin{proof}
		First, notice $H\Rbb$ is an $\einf$-ring spectrum and $C_2$-acts $H\Rbb$-linearly on the mapping spectrum \begin{equation*}
			\map\left(\Sigma_+^{\infty}\emb(\Fbb), \Sigma kr\wedge H\Rbb\right).
		\end{equation*} This implies the fixed point spectrum $ \map\left(\Sigma_+^{\infty}\emb(\Fbb), \Sigma kr\wedge H\Rbb\right)^{hC_2}$ is an $H\Rbb$-module spectrum. For any spectrum $X$ and an $H\Rbb$-module spectrum $Y$, we have an adjunction:
		\begin{equation*}
			\hom_{H\Rbb\text{-}\Mod}(X\wedge H\Rbb, Y)\simto \hom_{\Sp}(X,Y).
		\end{equation*}
		The map $R_\Fbb\wedge 1$ is the left adjoint of $R_\Fbb$ under this adjunction. Next, we compute homotopy groups of its cofiber. By \Cref{thm:QB_gal} and \Cref{thm:spectral_Borel_reg}, $\pi_{k}(R_\Fbb\wedge 1)$ is a $G$-isomorphism except for $k=0,1$. Denote the vector space $\Rbb$ with the trivial $G$-action by $\Rbb_\triv$.
		\begin{itemize}
			\item When $k=0$, the Galois group $G$ acts trivially on $K_0(\Ocal_\Fbb)\otimes_\Z \Rbb\cong \Rbb$. Notice \[\pi_0\left(\map\left(\Sigma_+^{\infty}\emb(\Fbb), \Sigma kr\wedge H\Rbb\right)^{hC_2}\right)=0,\] we have $\ker \pi_0(R_\Fbb\wedge 1)\cong \Rbb_\triv$.
			\item Recall from \eqref{eqn:hyperplane}, the Dirichlet regulator exhibits  $K_1(\Ocal_{\Fbb})\otimes_\Z \Rbb$ as the kernel of a $G$-equivariant map \[f\colon Y_0(\Fbb)\otimes_\Q \Rbb\to \Rbb_\triv.\]
			This implies $\cok \pi_1(R_\Fbb\wedge 1)\cong \Rbb_\triv$.  
		\end{itemize}
		The long exact sequence of homotopy groups associated to a cofiber sequence then yields $G$-isomorphisms:
		\begin{equation*}
			\pi_{*}\left(\mathrm{Cofib}~R_\Fbb\wedge 1\right)\cong \left\{\begin{array}{cl}
				\Rbb^{\oplus 2}_\triv, & *=1;\\
				0, & \text{else}.
			\end{array}\right.
		\end{equation*}
		Consequently, the cofiber of $R_\Fbb\wedge 1$ is $G$-equivalent to $\Sigma H(\Rbb^{\oplus 2})\simeq  \Sigma H\Rbb\vee \Sigma H\Rbb$ with the trivial $G$-action. 
	\end{proof}
	\begin{exmps}\label{exmp:emb_gal_c2}
		When $\Fbb/\Q$ is a $G$-Galois extension of number fields, we can give a more explicit description of the $G$-action on the spectrum $\map\left(\Sigma_+^{\infty}\emb(\Fbb), \Sigma kr\wedge H\Rbb\right)^{hC_2}$ in \Cref{prop:G-cofib}.  Under the Galois assumption, $G$ acts freely and transitively on $\emb(\Fbb)$.  This implies $\emb(\Fbb)\cong G$ as a $G$-set. Moreover, as a Galois extension of $\Q$, the field $\Fbb$ is either totally real or totally complex. 
		\begin{enumerate}
			\item If $\Fbb$ is totally real, the group $C_2$ acts on $\emb(\Fbb)$ trivially. This yields a $G$-equivalence of spectra:
			\begin{equation*}
				\map\left(\Sigma_+^{\infty}\emb(\Fbb), \Sigma kr\wedge H\Rbb\right)^{hC_2}\simeq \map\left(\Sigma_+^\infty G, \Sigma ko\wedge H\Rbb\right),
			\end{equation*}
			where $ko$ is the connective real topological $K$-theory spectrum with trivial $G$-action.
			\item If $\Fbb$ is totally complex, the group $C_2$ acts on $\emb(\Fbb)$ freely by \emph{post}-composing complex conjugation. This $C_2$-action coincides with the action of a subgroup $C_2\le G$ on $\emb(\Fbb)$ by \emph{pre}-composing complex conjugation. We then have a $G$-equivalence:
			\begin{equation*}
				\map\left(\Sigma_+^{\infty}\emb(\Fbb), \Sigma kr\wedge H\Rbb\right)^{hC_2}\simeq \map\left(\Sigma_+^\infty(G/C_2), \Sigma ku\wedge H\Rbb\right), 
			\end{equation*}
			where $ku$ is the connective complex topological $K$-theory spectrum with trivial $G$-action. 
		\end{enumerate}
		Combining \Cref{prop:G-cofib} and \Cref{exmp:emb_gal_c2}, we obtain an explicit description of the rational na\"ive $G$-homotopy type of the algebraic $K$-theory spectrum $K(\Ocal_{\Fbb})$, when $\Fbb/\Q$ is a $G$-Galois extension. 
	\end{exmps}
	Next, we connect rational algebraic $K$-theory spectra of number fields with Artin $L$-functions. To incorporate the Galois representation $\rho$ in the definition of Artin $L$-functions, we need to twist the algebraic $K$-theory spectrum by an equivariant Moore spectrum attached to $\rho$. 
	\begin{prop}[Kahn, {\cite[Corollary E]{Kahn_rational_moore}}]\label{prop:rational_eMoore}
		Let $G$ be a finite group and $V$ be a finite dimensional vector space over a field $\Lbb$ of characteristic $0$. Then any $\Lbb$-linear action on $V$ by $G$ can be uniquely lifted to a $G$-action on the Moore spectrum $M(V)\simeq HV$ by $H\Lbb$-module maps, such that the induced $G$-action on $H_0(M(V);\Z)$ is isomorphic to the prescribed $G$-action on $V$.
	\end{prop}
	\begin{notn}
		Denote by $M(\underline{\rho})$ the $G$-equivariant Moore spectrum attached to  $\rho\colon G\to\aut_\Lbb(V)$. 
	\end{notn}
	\begin{rem}
		The assignment $\rho\mapsto M(\underline{\rho})$ is functorial for rational $G$-representations. Let $\Bcal G$ be the groupoid with one object $*$ and morphism set $G$.  Then for any category $\Ccal$, a $G$-action on an object $c\in \Ccal$ can be regarded as a functor $\Bcal G\to \Ccal$ that sends $*$ to $c$. This applies to both $G$-representation/actions on $\Lbb$-vector spaces and on $H\Lbb$-module spectra.  By \Cref{prop:rational_ME}, rational Moore spectra are equivalent to the Eilenberg-MacLane spectra. The latter is a lax-monoidal functor $H\colon \Vect_\Lbb\to \Mod_{H\Lbb}$ that sends direct sums to wedge sums. The rational equivariant Moore spectrum construction is then a post-composition with the functor $H$:
		\begin{align*}
			\Fun(\Bcal G, \Vect_\Lbb)&\longrightarrow \Fun(\Bcal G,\Mod_{H\Lbb}),\\
		\underline{\rho}&\longmapsto M(\underline{\rho})\simeq  H\underline{\rho}.
		\end{align*}
		It follows that for any $G$-representations in $\Lbb$-vector spaces $\underline{\rho}$ and $\underline{\rho'}$, we have  equivalence  of \emph{na\"ive} $G$-equivariant $H\Lbb$-module spectra:
		\begin{equation*}
			M\left(\underline{\rho}\oplus \underline{\rho'}\right)\simeq M\left(\underline{\rho}\right)\vee M\left(\underline{\rho'}\right).
		\end{equation*}
		As the Eilenberg-MacLane spectrum construction is lax-monoidal, we have a natural $G$-equivariant map:
		\begin{equation*}
			 M\left(\underline{\rho}\right)\wedge M\left(\underline{\rho'}\right) \longrightarrow  M\left(\underline{\rho}\otimes \underline{\rho'}\right).
		\end{equation*}
		By the K\"unneth Theorem and the flatness of rational vector spaces as $\Z$-modules, the map above is an equivalence of \emph{na\"ive} rational $G$-spectra. 
			 
		In the \emph{genuine} equivariant setting, Schwede-Shipley showed in  \cite[Example 5.1.2]{Schwede-Shipley_stable_model} that "rational $G$-equivariant stable homotopy category is equivalent to the derived category of rational [$G$-]Mackey functors" for a finite group $G$. This implies if $G_1\le G_2$ are subgroups of $G$, then there are (essentially unique) functorial liftings of the restriction/induction maps between rational $G_1$- and $G_2$-representations to their rational equivariant Moore/Eilenberg-MacLane spectra.  
	\end{rem}

	\begin{lem}\label{lem:eUCT}
		Let $G$ be a finite group and $\rho\colon G\to\aut_\Lbb(V)$ be a $G$-representation over a field $\Lbb$ of characteristic $0$. For any spectrum $X$ with a $G$-action, we have an isomorphism of $G$-representations:
		\begin{equation*}
			\pi_*\left(X\wedge M\left(\underline{\rho}\right)\right)\cong \pi_*(X)\otimes \underline{\rho}.
		\end{equation*}
	\end{lem}
	\begin{proof}
		Non-equivariantly, this isomorphism follows from the Universal Coefficient Theorem and the flatness of $V$ as a $\Z$-module. To show it is $G$-equivariant, recall the Universal Coefficient Theorem is a special case of an Atiyah-Hirzebruch spectral sequence, where we view $\pi_*(X\wedge -)$ as a generalized homology theory:
		\begin{equation}\label{eqn:AHSS_UCT}
			E^2_{s,t}=H_s\left(M\left(\underline{\rho}\right);\pi_t(X)\right)\Longrightarrow \pi_{t+s}\left(X\wedge M\left(\underline{\rho}\right)\right)
.		\end{equation}
		The flatness of $V$ as a $\Z$-module and the Universal Coefficient Theorem for ordinary homology groups imply:
		\begin{equation*}
			H_s\left(M\left(\underline{\rho}\right);\pi_t(X)\right)\cong H_s\left(M\left(\underline{\rho}\right);\Z\right)\otimes_\Z\pi_t(X)\cong \left\{\begin{array}{cl}
				\underline{\rho}\otimes \pi_t(X), & s=0;\\
				0,&s\ne 0.
			\end{array}\right.
		\end{equation*}
		Consequently, the $E^2$-page of \eqref{eqn:AHSS_UCT} is concentrated in the $(s=0)$-line, and the spectral sequence collapses. From this we get a sequence of isomorphisms:
		\begin{equation*}
			\pi_t(X)\otimes \underline{\rho}\simto H_0\left(M\left(\underline{\rho}\right);\pi_t(X)\right)\simto\pi_t\left(X\wedge M\left(\underline{\rho}\right)\right).
		\end{equation*}
		We claim both isomorphisms are $G$-equivariant. For the first one, this is because the Universal Coefficient Theorem for ordinary homology is natural in both the space and the coefficient system. The second map is an edge homomorphism for the Atiyah-Hirzebruch spectral sequence \eqref{eqn:AHSS_UCT}. It is $G$-equivariant since the spectral sequence is natural in both the space and the generalized homology theory (representing spectrum). 
	\end{proof}
	\begin{rem}
		\Cref{lem:eUCT} holds more generally for a $G$-action $\rho$ on a flat $\Z$-module $A$, provided $\rho$ can be lifted to a $G$-action on the Moore spectrum $M(A)$.
	\end{rem}	
	\begin{proof}[Proof of \Cref{thm:main}]
		 By \Cref{prop:rational_ME} and \Cref{prop:rational_eMoore}, the equivariant Moore spectrum $M(\underline{\rho})$ is an $H\Lbb$-module spectrum with a $G$-action by $H\Lbb$-module maps. This implies both the fixed point subspaces $\left[\pi_{*}\left(K(\Ocal_\Fbb)\wedge M(\underline{\rho})\right)\right]^G$ and the equivariant homotopy groups $\pi_{*}\left[\left(K(\Ocal_\Fbb)\wedge M(\underline{\rho})\right)^{hG}\right]$ have natural $\Lbb$-vector space structures. Combining \Cref{thm:alg_QB_AL}, \Cref{lem:rational_FPT} and \Cref{lem:eUCT}, we have:
		\begin{align*}
			\ord_{s=1-n}L(s,\rho)&=\dim_\Lbb\left[K_{2n-1}(\Ocal_\Fbb)\otimes_\Z \underline{\rho}\right]^G\\
			&=\dim_\Lbb\left[\pi_{2n-1}\left(K(\Ocal_\Fbb)\wedge M(\underline{\rho})\right)\right]^G\\
			&=\dim_\Lbb\pi_{2n-1}\left[\left(K(\Ocal_\Fbb)\wedge M(\underline{\rho})\right)^{hG}\right]. \qedhere
		\end{align*} 
	\end{proof}
	
	\section{Further discussions: Integral equivariant Moore spectra associated to Galois representations}\label{sec:int_eMoore}
	Having studied the equivariant homotopy groups of $K(\Ocal_\Fbb)$ with coefficients in a \emph{rational} representation $\rho$, a natural question is if there is an \emph{integral} version of this story. The classical Quillen-Lichtenbaum Conjecture, proved by Voevodsky-Rost \cite{Voevodsky_motivic_EM,Haesemeyer-Weibel_norm_residue}, answers this question when $\rho$ is the trivial representation. 
	\begin{thm}[Quillen-Lichtenbaum Conjecture, Voevodsky-Rost, {\cite[199 -- 200]{Kolster_k_thy_arithmetic}}]\label{thm:QLC}
		Let $\Fbb$ be a number field. Denote by $\zeta^*_\Fbb(1-n)$ the leading coefficient in the Taylor expansion of the Dedekind zeta function $\zeta_\Fbb(s)$ at $s=1-n$. Then the following identity
		\begin{equation*}
			\zeta^*_{\Fbb}(1-n)=\pm\frac{|K_{2n-2}(\mathcal{O}_\Fbb)|}{|K_{2n-1}(\mathcal{O}_\Fbb)_{\mathrm{tors}}|}\cdot R^B_n(\Fbb)
		\end{equation*}
		holds up to powers of $2$, where the Borel regulator $R^B_n(\Fbb)$ is the covolume of the lattice $R_\Fbb^n(K_{2n-1}(\Ocal_\Fbb))\subseteq Y_{n-1}(\Fbb)\otimes_\Q\Rbb$ in \eqref{eqn:Borel_reg} ($\overline{Y_0(\Fbb)}\otimes_\Q\Rbb$ when $n=1$). 
	\end{thm} 
	Let $\Lbb$ be a number field and $\Fbb/\Kbb$ be a finite Galois extensions of number fields. By \cite[Proposition 9.3.5]{Diamond-Shurman}, any Galois representation $\rho\colon G=\gal(\Fbb/\Kbb)\to \GL_d(\Lbb)$ is similar to one that factors through $\GL_d(\Ocal_\Lbb)$. Denote this integral representation by $\Ocal_\rho\colon G\to \GL_d(\Ocal_\Lbb)$. To study equivariant $K$-theory of number fields with coefficients in \emph{integral} Galois representations, we need to lift the $G$-action on $\Ocal_\Lbb^{\oplus d}$ to the Moore spectrum $M(\Ocal_\Lbb^{\oplus d})$. 
	This is a special case of the following question of Steenrod: 
	\begin{quest}[Steenrod, {\cite[page 171]{Carlsson_counterexample}}]\label{quest:Steenrod}
		Let $A$ be an abelian group with a $G$-action. Is there a $G$-action on the Moore spectrum $M(A)$ such that the induced $G$-action on $H_0(M(A);\Z)\cong A$ is isomorphic to the prescribed $G$-action on $A$?
	\end{quest}
	While Steenrod's question does not always have positive answers (see Carlsson's counterexamples in \cite[Theorem 2]{Carlsson_counterexample}), equivariant Moore spectra associated to abelian characters can be constructed explicitly. 
	\begin{thm}[{\cite[Section 3.3]{nz_Dirichlet_J}}]\label{thm:eMoore}
		Let $\psi_m\colon  C_m\to \Z[\zeta_m]^\times$ be the group homomorphism that sends a generator of the cyclic group $C_m$ to a primitive $m$-th root of unity $\zeta_m$. Then the Moore spectrum $M(\Z[\zeta_m])$ has a finite $C_m$-CW spectrum structure, lifting the action of $C_m$ on $\Z[\zeta_m]$ by $\psi_m$ in the sense of Steenrod's \Cref{quest:Steenrod}. 
	\end{thm}
	\begin{cor}\label{cor:eMoore_char}
		Let $\chi\colon  G\to \Cx$ be an abelian character of a finite group $G$. Then there is a $G$-equivariant integeral Moore spectrum $M(\underline{\Ocal_\chi})$.  
	\end{cor}
	\begin{proof}
		Notice the abelian character $\chi$ factors as 
		\begin{equation}\label{eqn:ab_char_factor}
			\begin{tikzcd}
				\chi\colon G\rar["\phi_\chi",->>]\ar[rr,bend right=30,"\Ocal_\chi"]&C_m\rar["\psi_m"]& (\Z[\zeta_m])^\x\rar[hook]&\Cx,
			\end{tikzcd}
		\end{equation}
		for some unique integer $m$. We can then set $M(\underline{\Ocal_\chi})$ by restricting the $C_m$-action on the Moore spectrum $M(\Z[\zeta_m])$ in \Cref{thm:eMoore} to $G$ via $\phi_\chi$.
	\end{proof}
	\begin{rem}\label{rem:eMoore_Q}
		Similar to \eqref{eqn:ab_char_factor}, an abelian character $\chi\colon G\to \Cx$ of a finite group $G$ also factors through $\Q_\chi\colon  G\to \Q(\zeta_m)^\x$ for some $m$. Notice the associated Galois representation $\underline{\Q_\chi}$ is $G$-isomorphic to $\underline{\Ocal_\chi}\otimes_\Z \Q$. The uniqueness part of \Cref{prop:rational_eMoore} then implies a rational $G$-equivalence of equivariant Moore spectra:
		\begin{equation*}
			M(\underline{\Ocal_\chi})\wedge S^0_\Q\simeq M(\underline{\Q_\chi}),
		\end{equation*}
		for any model of an integral equivariant Moore spectrum $M(\underline{\Ocal_\chi})$ attached to $\chi$.
	\end{rem}
	\begin{rems} Compared with the nice properties of \emph{rational} (equivariant) Moore spectra in \Cref{prop:rational_ME} and \Cref{prop:rational_eMoore}, \emph{integral} (equivariant) Moore spectra have the following "defects":
		\begin{enumerate}
			\item While integral equivariant Moore spectra exist for abelian characters of finite groups, there are non-equivalent $G$-actions on the Moore spectrum inducing the same action on the homology groups. For example, consider $C_2$-representation spheres of the form $S^{2k(\sigma-1)}$, where $\sigma$ is the real sign representation and $k\ge 0$. Then the induced $C_2$-actions on their zeroth homology groups are all trivial. 
			\item For a number field $\Lbb$, the Moore spectrum $M(\Ocal_\Lbb)$ of its ring of integers does not have an $E_\infty$-ring spectrum structure in general. For example, in \cite{SVW_rou_Einf}, Schw\"{a}nzl-Vogt-Waldhausen showed that there is no way to "adjoining $\sqrt{-1}$" to the sphere spectrum as an $E_\infty$-ring spectrum. This means the Moore spectrum $M(\Z[i])$ of the Gaussian integers $\Z[i]$ does not admit an $E_\infty$-ring spectrum structure.  
		\end{enumerate}
	\end{rems}
	By \Cref{cor:eMoore_char}, integral equivariant Moore spectra $M(\underline{\Ocal_\chi})$ attached to Dirichlet characters $\chi\colon \znx\cong\gal(\Q(\zeta_N)/\Q)\to \Cx$ always exist (though not uniquely). In my thesis \cite{nz_Dirichlet_J}, I computed equivariant homotopy groups of the $J$-spectra $J(N)$ with coefficients in the character $\Ocal_\chi$. These equivariant homotopy groups are related them with the denominators of special values of Dirichlet $L$-functions. In my 
	current work in progress with Elden Elmanto
	, we are studying equivariant algebraic $K$-groups of $\Z[\zeta_N]$ with coefficients in the integral Dirichlet character $\Ocal_\chi$ of the form:
	\begin{equation}\label{eqn:twisted_alg_K}
		\pi_*\left(K(\Z[\zeta_N])\wedge M\left(\underline{\Ocal_\chi}\right)\right)^{\znx}.
	\end{equation}
	 As Dirichlet $L$-functions are special cases of Artin $L$-functions, \Cref{thm:main} and \Cref{rem:eMoore_Q} imply:
	\begin{cor}
		Suppose the image of a Dirichlet character $\chi\colon \znx\to \Cx$ is cyclic of order $m$. Then we have
		\begin{equation*}
			\dim_{\Q(\zeta_m)}\left[\pi_{2n-1}\left(K(\Z[\zeta_N])\wedge M\left(\underline{\Ocal_\chi}\right)\right)^{h \znx}\otimes \Q\right]=\ord_{s=1-n} L(s,\chi).
		\end{equation*}
	\end{cor}
	We hope to generalize the Quillen-Lichtenbaum Conjecture \Cref{thm:QLC} to Dirichlet $L$-functions by computing the torsion subgroups of equivariant algebraic $K$-groups of $\Z[\zeta_N]$ with coefficients in $\chi$ in \eqref{eqn:twisted_alg_K}. 
	\begin{rem}
		One might further wonder about a potential Quillen-Lichtenbaum Conjecture for Artin $L$-functions. However, it is not clear whether integral equivariant Moore spectra attached to Galois representations $\Ocal_\rho\colon G=\gal(\Fbb/\Kbb)\to \GL_d(\Ocal_\Lbb)$ exist or not when $d>1$.
		
		One attempt is to use the Brauer Induction Theorem \cite[Theorem 20]{Serre_rep_fin_gp}, which states that for a finite group $G$, its complex representation ring $R(G)$ is generated as an abelian group by inductions of abelian characters on subgroups. When $\rho$ is a direct sum of inductions of abelian characters on subgroups,  we can construct a $G$-equivariant integral Moore spectrum  $M(\underline{\Ocal_\rho})$ associated to $\rho$ by bootstrapping the equivariant Moore spectrum $M(\underline{\Ocal_\chi})$ in \Cref{cor:eMoore_char}. When $\rho$ is a virtual difference of two sums of inductions of abelian characters on subgroups, we do not know whether $M(\underline{\Ocal_\rho})$ exists or not. 
	\end{rem}

	\printbibliography
\end{document}